\documentclass[11pt,reqno]{amsart}

\usepackage[a4paper,top=3cm,bottom=3cm,left=2.7cm,right=2.7cm]{geometry}

\usepackage{amssymb}
\usepackage{amsmath}
\usepackage{mathtools}
\usepackage{mathrsfs}
\usepackage{enumerate}
\usepackage{xcolor}

\usepackage[colorlinks=true,
            citecolor=red,
            linkcolor=blue,
            urlcolor=red,
            bookmarksnumbered=true,
            bookmarksopen=true]{hyperref}
\usepackage[capitalize,noabbrev]{cleveref}

\crefname{thm}{Theorem}{Theorems}
\Crefname{thm}{Theorem}{Theorems}
\crefname{problem}{Problem}{Problems}
\Crefname{problem}{Problem}{Problems}
\crefname{question}{Question}{Questions}
\Crefname{question}{Question}{Questions}

\numberwithin{equation}{section}

\newtheorem{thm}{Theorem}[section]
\newtheorem{theorem}[thm]{Theorem}

\newtheorem{question}[thm]{Question}
\newtheorem{problem}[thm]{Problem}

\newtheorem*{problem*}{Problem}
\newtheorem*{question*}{Question}

\theoremstyle{definition}

\begin{document}

\title[On some open problems in commutative algebra resolved by Rethlas]{On some open problems in commutative algebra resolved by Rethlas}

\author{Jiedong Jiang, Yixiao Li, Zeming Sun, Yuefeng Wang, Liang Xiao, and Jiahong Yu}

\address{Jiedong Jiang, Westlake Institute for Advanced Study, Westlake University, No. 600 Dunyu Road, Sandun town, Xihu district, Hangzhou, Zhejiang, 310030, China.}
\email{jiangjiedong@westlake.edu.cn}

\address{Yixiao Li, Beijing International Center for Mathematical Research, Peking University, No.~5 Yiheyuan Road, Haidian District, Beijing 100871, China.}
\email{2201110016@pku.edu.cn}

\address{Zeming Sun, Beijing International Center for Mathematical Research, Peking University, No.~5 Yiheyuan Road, Haidian District, Beijing 100871, China.}
\email{zeming@kurims.kyoto-u.ac.jp}

\address{Yuefeng Wang, Beijing International Center for Mathematical Research, Peking University, No.~5 Yiheyuan Road, Haidian District, Beijing 100871, China.}
\email{2501110002@stu.pku.edu.cn}

\address{Liang Xiao, New Cornerstone Science Laboratory, School of Mathematical Sciences, Peking University, No.~5 Yi He Yuan Road, Haidian District, Beijing, 100871, China.}
\email{lxiao@bicmr.pku.edu.cn}

\address{Jiahong Yu, Morningside Center of Mathematics, Chinese Academy of Sciences, No.~55, Zhongguancun East Road, Beijing, 100190, China.}
\email{yu\_jh@amss.ac.cn}

\date{\today}

\begin{abstract}
We report on a collection of open problems in commutative algebra and related areas
that have been resolved (proved or disproved) using the Rethlas
natural-language automated reasoning system. The problems are drawn from several
published lists, including \emph{Open Problems in Commutative Ring Theory}
(Cahen--Fontana--Frisch--Glaz), Erman--Sam's survey of Boij--S\"oderberg
theory.
For each problem we record the precise statement and a
self-contained proof produced (with no human intervention) by Rethlas
and subsequently verified by human experts.
\end{abstract}

\maketitle

\tableofcontents

\section{Introduction}

This paper records solutions to a collection of open problems in
commutative algebra and related areas. The problems are drawn from two
sources: the volume \emph{Open Problems in Commutative Ring Theory}
edited by Cahen, Fontana, Frisch and Glaz~\cite{CFFG2014}, and the
survey of Boij--S\"oderberg theory by Erman and
Sam~\cite{ErmanSam2017}. All proofs were generated by the Rethlas
automated reasoning system~\cite{JuEtAl2026Rethlas}
\footnote{The raw output of Rethlas is available at
\href{https://github.com/frenzymath/Rethlas_results/tree/main/analysis/Brezis_OpenProblems}{Rethlas results homepage}.} 
and subsequently
verified by human experts. We summarize the results below, organized
by mathematical subject.

Two of the problems concern the behaviour of GCD-like ring conditions
under group-ring constructions.
Glaz~\cite{Glaz2000-FCRZD,Glaz2000-FCR} studied the notations of
G-GCD rings and finite conductor rings, extending the classical theory of GCD domains.
Problems~\ref{prob:4ai} and~\ref{prob:4aii} ask whether the finite
conductor and quasi coherent properties ascend from a G-GCD ring~$R$ to
the group ring~$RG$ for a finitely generated abelian group~$G$. Both
answers are negative, via counterexamples built from Glaz's non-coherent
UFD of characteristic~$2$ in \cite{Glaz2000-FCRZD}.

Closely related is Problem~\ref{prob:35}, which asks whether an
almost GCD (AGCD) domain in which every nonzero $t$-locally principal
ideal is $t$-invertible must be of finite $t$-character. Unlike
Problems~\ref{prob:4ai} and~\ref{prob:4aii}, the answer here
is affirmative, proved using results of Zafrullah on
$t$-invertibility~\cite{Zafrullah-UniversalRestriction,Zafrullah-BazzoniLike}.

Problem~\ref{prob:8a} belongs to the theory of completions of Noetherian local rings.
Anderson~\cite{Anderson2014-QCRM} proved that a ring is quasi-complete if and only if
every homomorphic image is weakly quasi-complete, and asked whether the two properties
actually coincide. We construct a counterexample, thereby answering the question in the
negative. This proof was later auto-formalized in Lean~4 by the Archon system and is
thus formally machine-checked; see~\cite{JuEtAl2026Rethlas}.

Problem~\ref{prob:21} concerns the algebraic structure of rings of
integer-valued polynomials, see~\cite{CahenChabert1997} for background
on $\operatorname{Int}(D)$.
Elliott~\cite{Elliott2010-Plethories} showed that
$\operatorname{Int}(D)$ admits a $D$-$D$-biring structure compatible
with $D[X]\hookrightarrow\operatorname{Int}(D)$ if and only if the
canonical maps
$\operatorname{Int}(D)^{\otimes_D n}\to\operatorname{Int}(D^n)$ are
isomorphisms for all~$n$. Problem~\ref{prob:21} asks whether such a
biring structure always exists. We show it does not by providing a counterexample.

Problem~\ref{prob:37b} asks whether David's $3$-dimensional
non-Noetherian factorial domain~$J$~\cite{David1973} is locally Jaffard.
Dobbs, Fontana, and Kabbaj~\cite{DobbsFontanaKabbaj1990} showed that $J$
is Jaffard, but the locally Jaffard property does not follow formally
from this. We give an affirmative
proof showing that $J$ is indeed locally Jaffard.

Finally, we address two questions on Boij--S\"oderberg
theory~\cite{BoijSoderberg2008,EisenbudSchreyer2009}. A natural problem
in this theory is to determine which integral points on a pure ray of the
Boij--S\"oderberg cone are realizable as Betti table of some module.
Question~\ref{q:bs61} asks whether every such point (in codimension~$3$)
is realized over either $S=k[x,y,z]$ or the enveloping algebra of the
Heisenberg Lie algebra; Question~\ref{q:bs62} asks the analogous question
for arbitrary finite-dimensional positively graded Lie algebras generated
in degree~$1$. Both answers are negative, via constructing explicit degree sequences
whose primitive integral points cannot be realized.

\subsection*{Acknowledgements} The authors would like to thank Shiji Lyu for verifying the proof of Anderson's open problem.
They also thank Haocheng Ju, Shurui Liu, Guoxiong Gao, Leheng Chen, Bin Wu and
Bin Dong for their contributions to the Rethlas project
\cite{JuEtAl2026Rethlas}.

This work is supported in part by the National Key R\&D Program of China grant 2024YFA1014000, the Fundamental and Interdisciplinary Disciplines Breakthrough Plan of the Ministry of Education of China (JYB2025XDXM113), and the New Cornerstone Investigator Program.

\section{Statements of resolved problems}\label{sec:statements}

We list below the open problems resolved by Rethlas, organised by source.
The corresponding proofs appear in \Cref{sec:proofs}.

\subsection{Problems from \emph{Open Problems in Commutative Ring Theory}}
\label{subsec:opcrt}

The reference is the volume \emph{Open Problems in Commutative Ring Theory}
edited by Cahen, Fontana, Frisch and Glaz \cite{CFFG2014}. We retain the
problem numbering of \emph{loc.\ cit.}

\begin{problem}[Problem 4a(i) of \cite{CFFG2014}]\label{prob:4ai}
A ring $R$ is a \emph{finite conductor ring} if $aR\cap bR$ and
$(0:c)$ are finitely generated ideals of $R$ for all $a,b,c\in R$, and a
\emph{G-GCD ring} if principal ideals of $R$ are projective and the
intersection of two finitely generated flat ideals is a finitely generated
flat ideal. Assume that $R$ is a G-GCD ring and $G$ is a finitely generated
abelian group. Does the finite conductor property ascend from $R$ to the
group ring $RG$?
\end{problem}

\begin{problem}[Problem 4a(ii) of \cite{CFFG2014}]\label{prob:4aii}
A ring $R$ is a \emph{quasi coherent ring} if
$a_1 R\cap\cdots\cap a_n R$ and $(0:c)$ are finitely generated ideals of $R$
for all $a_1,\dots,a_n,c\in R$, and a \emph{G-GCD ring} if principal ideals
of $R$ are projective and the intersection of two finitely generated flat
ideals is a finitely generated flat ideal. Assume that $R$ is a G-GCD ring
and $G$ is a finitely generated abelian group. Does the quasi coherent
property ascend from $R$ to the group ring $RG$?
\end{problem}

\begin{problem}[Problem 8a of \cite{CFFG2014} (Anderson's conjecture)]\label{prob:8a}
Let $(R,\mathfrak m)$ be a Noetherian local ring. $R$ is said to be
\emph{quasi-complete} if for every decreasing sequence
$\{I_n\}_{n=1}^\infty$ of ideals of $R$ and every natural number $k$,
there exists $s_k$ with
$I_{s_k}\subseteq \bigl(\bigcap_{n=1}^\infty I_n\bigr)+\mathfrak m^k$;
if this condition is required only for sequences with
$\bigcap_n I_n=(0)$, then $R$ is called \emph{weakly quasi-complete}.
It is known that $R$ is quasi-complete if and only if each homomorphic
image of $R$ is weakly quasi-complete \cite{Anderson2014-QCRM}. Prove
that there exists a weakly quasi-complete ring that is not quasi-complete.
\end{problem}

\begin{problem}[Problem 21 of \cite{CFFG2014}]\label{prob:21}
For an integral domain $D$ with field of fractions $K$, let
\[
   \operatorname{Int}(D)
   :=\{f\in K[X]\mid f(D)\subseteq D\}
\]
denote the ring of integer-valued polynomials on $D$. By
\cite{Elliott2010-Plethories} the existence of a $D$-$D$-biring structure
on $\operatorname{Int}(D)$ such that $D[X]\to\operatorname{Int}(D)$ is a
biring homomorphism is equivalent to the canonical maps
$\operatorname{Int}(D)^{\otimes_D n}\to\operatorname{Int}(D^n)$ being
isomorphisms for all $n$. Does $\operatorname{Int}(D)$ always admit a
unique $D$-$D$-biring structure such that $D[X]\to\operatorname{Int}(D)$
is a homomorphism of $D$-$D$-birings?
\end{problem}

\begin{problem}[Problem 35 of \cite{CFFG2014}]\label{prob:35}
An integral domain $D$ is an \emph{almost GCD (AGCD) domain} if for every
pair $x,y\in D\setminus\{0\}$ there is an integer $n=n(x,y)\ge 1$ with
$x^n D\cap y^n D$ principal. $D$ is of \emph{finite $t$-character} if every
nonzero non-unit of $D$ lies in at most finitely many maximal $t$-ideals.
An ideal $I$ of $D$ is \emph{$t$-locally principal} if $ID_P$ is principal
for every maximal $t$-ideal $P$ of $D$. Let $D$ be an almost GCD domain
such that every nonzero $t$-locally principal ideal of $D$ is
$t$-invertible. Is $D$ of finite $t$-character?
\end{problem}

\begin{problem}[Problem 37b of \cite{CFFG2014}]\label{prob:37b}
A finite-dimensional integral domain $D$ is called a \emph{Jaffard domain}
if $\dim(D[X_1,\dots,X_n])=n+\dim(D)$ for all $n\ge 1$; equivalently, if
$\dim(D)=\dim_v(D)$, where $\dim_v$ denotes the valuative dimension. $D$
is \emph{locally Jaffard} if $D_P$ is Jaffard for every prime ideal $P$ of
$D$. Let $k$ be a field of characteristic zero, let $\{s(n)\}_{n\ge 2}$ be
a sequence of positive integers, and let
\[
   J=\bigcup_{n\ge 1} J_n,\qquad J_n=k[X,\zeta_{n-1},\zeta_n],
\]
where the $\zeta_n$ satisfy the recurrence
\[
   \zeta_n=\frac{\zeta_{n-1}^{s(n)}+\zeta_{n-2}}{X}\qquad(n\ge 2),
\]
so that $J_n\subseteq J\subseteq J_n[X^{-1}]$ for every $n$. This is
David's 3-dimensional factorial domain \cite{David1973}. By
\cite{DobbsFontanaKabbaj1990}, $J$ is Jaffard. Is $J$ locally Jaffard?
\end{problem}

\subsection{Questions on Boij--S\"oderberg theory}
\label{subsec:bs}

The reference is the survey of Erman and Sam
\cite{ErmanSam2017}.

\begin{question}[Question 6.1 of \cite{ErmanSam2017}]\label{q:bs61}
Let $k$ be a field, let $S=k[x,y,z]$ with the standard grading, and let
$\mathfrak H$ be the Heisenberg Lie algebra with basis $\{x,y,z\}$,
bracket $[x,y]=z$, $[x,z]=[y,z]=0$, graded by $\deg(x)=\deg(y)=1$,
$\deg(z)=2$. For every degree sequence $(d_0,d_1,d_2,d_3)$ and every
integral point on the corresponding pure ray in the Boij--S\"oderberg
cone, does there exist a finite length graded module over $S$ or over
$U(\mathfrak H)$ whose Betti table is that integral point?
\end{question}

\begin{question}[Question 6.2 of \cite{ErmanSam2017}]\label{q:bs62}
For every degree sequence $(d_0,\dots,d_n)$ and every integral point on
the corresponding pure ray in the Boij--S\"oderberg cone, does there exist
an $n$-dimensional $\mathbb Z_{>0}$-graded Lie algebra $\mathfrak g$
generated in degree~$1$ over a field $k$, together with a finite length
graded module $M$ over $U(\mathfrak g)$, whose Betti table is that
integral point?
\end{question}

\section{Proofs}\label{sec:proofs}

In this section we give the proofs of the problems stated in
\Cref{sec:statements}. Each subsection treats a single problem: we first
restate the problem for the reader's convenience, then give the proof,
introducing any auxiliary lemmas as needed.

\subsection{Solution to Problem~\ref{prob:4ai}}\label{subsec:proof-4ai}

A ring $R$ is a \emph{finite conductor ring} if $aR\cap bR$ and the
annihilator $(0:c):=\{r\in R:rc=0\}$ are finitely generated ideals of
$R$ for every $a,b,c\in R$; it is a \emph{G-GCD ring} if its principal
ideals are projective and the intersection of two finitely generated
flat ideals is again a finitely generated flat ideal.

\begin{theorem}\label{thm:4ai}
There exists a G-GCD ring $R$ of characteristic $2$ (in particular,
a UFD) and a finitely generated abelian group $G$ such that the group
ring $RG$ is not a finite conductor ring. Consequently, the finite
conductor property does not ascend from G-GCD rings to group rings over
finitely generated abelian groups.
\end{theorem}

\begin{proof}
The argument has three steps.

\smallskip
\noindent\emph{Step 1: a non-coherent UFD of characteristic $2$.}
By Example~9 of \cite{Glaz2000-FCRZD}, the local ring
\[
   A:=S^{\langle\sigma\rangle},\qquad
   S=\mathbf F_2(\{\alpha_i\},\{\beta_i\})[x,y]_{(x,y)},
\]
where $\sigma$ is the order-$2$ automorphism of $S$ defined by
$\sigma(x)=x$, $\sigma(y)=y$, $\sigma(\alpha_i)=\alpha_i+y\rho_{i+1}$,
$\sigma(\beta_i)=\beta_i+x\rho_{i+1}$ with
$\rho_i=\alpha_i x+\beta_i y$, is a local Krull domain of Krull
dimension $2$, is \emph{not} coherent, and is a UFD. In particular $A$
is a UFD, hence a GCD domain, hence a G-GCD domain, hence a finite
conductor ring, of characteristic $2$.

\smallskip
\noindent\emph{Step 2: bad ideal in $A$.}
Since $A$ is a non-coherent domain, some finitely generated ideal of $A$
is not finitely presented (all annihilator ideals $(0:c)$ vanish because
$A$ is a domain). Choose such an ideal
\[
   I=(a_1,\dots,a_m)\subseteq A,
\]
and let
\[
   \varphi\colon A^m\to A,\qquad
   \varphi(\xi_1,\dots,\xi_m)=\sum_{i=1}^m a_i\xi_i,
\]
so that the exact sequence $0\to K\to A^m\to I\to 0$ with $K=\ker\varphi$
shows that $K$ cannot be finitely generated over $A$, since $I$ is not
finitely presented.

\smallskip
\noindent\emph{Step 3: the syzygy is detected in $A[(\mathbf Z/2\mathbf Z)^m]$.}
Set $H=(\mathbf Z/2\mathbf Z)^m$ and $B=A[H]$, with standard generators
$g_1,\dots,g_m$ of $H$. Put $\varepsilon_i:=g_i+1\in B$. Because
$\operatorname{char}(A)=2$ and $g_i^2=1$, we have $\varepsilon_i^2=0$
and the $\varepsilon_i$ commute. The monomials
\[
   \varepsilon_T:=\prod_{i\in T}\varepsilon_i \qquad (T\subseteq\{1,\dots,m\})
\]
form an $A$-basis of $B$, and the resulting decomposition
\[
   B=\bigoplus_{r=0}^m B_r,\qquad
   B_r=\bigoplus_{|T|=r} A\,\varepsilon_T,
\]
is a $\mathbf Z_{\ge0}$-grading. Define
\[
   f:=\sum_{i=1}^m a_i\varepsilon_i\in B_1.
\]
Since multiplication by $f$ raises degree by~$1$, the annihilator
$\operatorname{Ann}_B(f):=\{r\in B:rf=0\}$ is a graded ideal of $B$.
Setting
$w:=\varepsilon_1\cdots\varepsilon_m\in B_m$ and
$v_i:=\prod_{j\ne i}\varepsilon_j\in B_{m-1}$ one has $\varepsilon_i v_i=w$
and $\varepsilon_j v_i=0$ for $j\ne i$, so
\[
   f\Bigl(\sum_{i=1}^m \xi_i v_i\Bigr)=\Bigl(\sum_{i=1}^m a_i\xi_i\Bigr)\,w.
\]
Hence the degree-$(m-1)$ component of $\operatorname{Ann}_B(f)$ is
naturally isomorphic to $K=\ker(\varphi)$. If $\operatorname{Ann}_B(f)$
were finitely generated as an ideal of $B$, then (since $B$ is a
finite free $A$-module) it would be finitely generated as an
$A$-module, and (since it is graded) each of its homogeneous components
would be finitely generated as an $A$-module --- in particular the
degree-$(m-1)$ component, isomorphic to $K$, would be finitely
generated, contradicting Step~2. Therefore $\operatorname{Ann}_B(f)$ is
not finitely generated, and $B=A[H]$ is not a finite conductor ring.

\smallskip
Thus $A$ is a G-GCD ring, the finite abelian group $H$ is finitely
generated, and $A[H]$ is not a finite conductor ring; so the finite
conductor property does not ascend from $R$ to $RG$ in general.
\end{proof}

\subsection{Solution to Problem~\ref{prob:4aii}}\label{subsec:proof-4aii}

A ring $R$ is a \emph{quasi coherent ring} if $a_1R\cap\cdots\cap a_n R$
and the annihilator $(0:c):=\{r\in R:rc=0\}$ are finitely generated
ideals of $R$ for every $a_1,\dots,a_n,c\in R$ (in particular, every
quasi coherent ring is a finite conductor ring). $R$ is a
\emph{G-GCD ring} if its principal ideals are projective and the
intersection of two finitely generated flat ideals is again a finitely
generated flat ideal.

\begin{theorem}\label{thm:4aii}
There exists a G-GCD ring $R$ of characteristic $2$ (in particular,
a UFD) and a finitely generated abelian group $G$ such that the group
ring $RG$ is not a quasi coherent ring. Consequently, the quasi
coherent property does not ascend from G-GCD rings to group rings over
finitely generated abelian groups.
\end{theorem}

\begin{proof}
Let $A$ be the local non-coherent UFD of
characteristic~$2$ supplied by Example~9 of \cite{Glaz2000-FCRZD}: with
\[
   F=\mathbf F_2(\{\alpha_i\},\{\beta_i\}),\qquad S=F[x,y]_{(x,y)},
\]
and the order-$2$ automorphism $\sigma$ of $S$ given by $\sigma(x)=x$,
$\sigma(y)=y$, $\sigma(\alpha_i)=\alpha_i+y\rho_{i+1}$,
$\sigma(\beta_i)=\beta_i+x\rho_{i+1}$ with
$\rho_i=\alpha_i x+\beta_i y$, the fixed ring
$A:=S^{\langle\sigma\rangle}$ is a local Krull domain of Krull
dimension~$2$, is a UFD, and is not coherent. Every UFD is a GCD domain
and hence a G-GCD domain (the intersection of two invertible ideals of a
UFD is principal, hence invertible), so $A$ is a G-GCD ring.

Since $A$ is not coherent, choose a finitely generated ideal
$I=(a_1,\dots,a_m)\subseteq A$ whose syzygy module
\[
   K:=\ker\Bigl(A^m\to A,\ (\xi_1,\dots,\xi_m)\mapsto\sum_{i=1}^m a_i\xi_i\Bigr)
\]
is not finitely generated; this is possible because $A$ is a domain (so
all $(0:c)$ vanish) and the failure of coherence therefore comes from
some finitely generated ideal not being finitely presented.

Set $G=(\mathbf Z/2\mathbf Z)^m$ and $B=A[G]$; then $G$ is a finite,
hence finitely generated, abelian group. With $g_1,\dots,g_m$ the
standard generators of $G$, put $\varepsilon_i=g_i+1\in B$. Since
$\operatorname{char}(A)=2$ and $g_i^2=1$, we have $\varepsilon_i^2=0$
and the $\varepsilon_i$ commute, so the monomials
$\varepsilon_T:=\prod_{i\in T}\varepsilon_i$ ($T\subseteq\{1,\dots,m\}$)
form an $A$-basis of $B$ and the decomposition
$B=\bigoplus_{r=0}^m B_r$ by $r=|T|$ is a $\mathbf Z_{\ge 0}$-grading.
Define
\[
   f:=\sum_{i=1}^m a_i\varepsilon_i\in B_1.
\]
Multiplication by $f\in B_1$ raises grading degree by $1$, so the
annihilator $\operatorname{Ann}_B(f):=\{r\in B:rf=0\}$ is a graded
ideal. With
$w=\varepsilon_1\cdots\varepsilon_m$ and $v_i=\prod_{j\ne i}\varepsilon_j$
one has $\varepsilon_i v_i=w$ and $\varepsilon_j v_i=0$ for $j\ne i$, so
\[
   f\Bigl(\sum_{i=1}^m \xi_i v_i\Bigr)=\Bigl(\sum_{i=1}^m a_i\xi_i\Bigr)w,
\]
which exhibits the degree-$(m-1)$ component of $\operatorname{Ann}_B(f)$ as
naturally isomorphic to $K$. If $\operatorname{Ann}_B(f)$ were
finitely generated as an ideal of $B$, then (since $B$ is a finite free
$A$-module) it would be finitely generated over $A$, and (since it is
graded) each of its homogeneous components --- in particular the
degree-$(m-1)$ component $K$ --- would be finitely generated over $A$,
contradicting the choice of $K$. Hence $\operatorname{Ann}_B(f)$ is
not finitely generated.

Thus $B$ is not a finite conductor ring. By definition every quasi
coherent ring is a finite conductor ring, so $B=A[G]$ is not quasi
coherent. Therefore the quasi coherent property does not ascend from
$R$ to $RG$ in general.
\end{proof}

\subsection{Solution to Problem~\ref{prob:8a}}\label{subsec:proof-8a}

Let $(R,\mathfrak m)$ be a Noetherian local ring, with $\mathfrak m$-adic
completion $\widehat R$. $R$ is \emph{quasi-complete} if for every
decreasing sequence $\{I_n\}_{n=1}^\infty$ of ideals of $R$ and every
natural number $k$ there exists $s_k$ with
$I_{s_k}\subseteq(\bigcap_n I_n)+\mathfrak m^k$; if this condition is
required only for sequences with $\bigcap_n I_n=(0)$, $R$ is called
\emph{weakly quasi-complete}. By \cite{Anderson2014-QCRM}, $R$ is
quasi-complete if and only if every homomorphic image of $R$ is weakly
quasi-complete. When $R$ is a domain, its \emph{generic formal fiber}
is the localization $\widehat R\otimes_R\operatorname{Frac}(R)$; a
Noetherian local domain is \emph{analytically irreducible} if its
completion is again a domain.

\begin{theorem}\label{thm:8a}
There exists a Noetherian local ring that is weakly quasi-complete but
not quasi-complete.
\end{theorem}

\begin{proof}
We exhibit a $2$-dimensional local UFD $A$ such that $A$ is weakly
quasi-complete but some homomorphic image of $A$ is not weakly
quasi-complete. By the criterion that $R$ is quasi-complete if and only
if every homomorphic image of $R$ is weakly quasi-complete
\cite{Anderson2014-QCRM}, such an $A$ is not quasi-complete.

\smallskip
\noindent\emph{Step 1: choice of a complete target ring.} Let
\[
   T:=\mathbf C[[x,y,z]]/(x^2-yz),\qquad
   \mathfrak m:=(x,y,z)T.
\]
The substitution $x\mapsto uv$, $y\mapsto u^2$, $z\mapsto v^2$
identifies $T$ with the subring $\mathbf C[[u^2,uv,v^2]]\subset
\mathbf C[[u,v]]$, so $T$ is a domain. Since $x^2-yz$ is a non-zero
divisor in the regular local ring $\mathbf C[[x,y,z]]$, the quotient
$T$ is a $2$-dimensional Cohen--Macaulay hypersurface, in particular a
complete $2$-dimensional Cohen--Macaulay local domain. Also
$T/\mathfrak m\cong\mathbf C$ and a formal power series ring in finitely
many variables over an uncountable field has the same cardinality as the
field, so $|T|=|T/\mathfrak m|=|\mathbf C|$. Finally, the height-one
prime $Q=(x,y)T$ is not principal: $T/Q\cong\mathbf C[[z]]$, and the
images of $x$ and $y$ in $Q/\mathfrak m Q$ are $\mathbf C$-linearly
independent (because $x,y\notin\mathfrak m^2$, the defining relation
being quadratic), so the minimal number of generators
$\mu(Q):=\dim_{T/\mathfrak m}(Q/\mathfrak m Q)\ge 2$.

\smallskip
\noindent\emph{Step 2: a local UFD with completion $T$ and trivial generic
formal fiber.} By \cite[Corollary~2.4]{Jensen2006-UFD}, there exists a
local UFD $A$ such that $\widehat A=T$ and the generic formal fiber of
$A$ is local with maximal ideal $P=(0)$, provided $(T,\mathfrak m)$ is
a complete local ring with $|T|=|T/\mathfrak m|$, $T$ has depth at least
two, $P=(0)$ is non-maximal and contains all associated primes of $T$
and meets the prime subring trivially, and for every
$J\in\operatorname{Spec}T$ with $\operatorname{ht}(J)>\operatorname{depth}(T_J)=1$ one
has $J\subseteq P$. All these hypotheses hold for the ring $T$ of Step~1
and $P=(0)$: $T$ is a complete Cohen--Macaulay domain of depth $2$, the
unique associated prime is $(0)\subset P$, and the Cohen--Macaulay
property excludes any prime $J$ with
$\operatorname{ht}(J)>\operatorname{depth}(T_J)=1$. Hence a local UFD $A$
with $\widehat A\cong T$ and local generic formal fiber with maximal ideal
$(0)$ exists.

\smallskip
\noindent\emph{Step 3: $A$ is weakly quasi-complete.} Because the generic
formal fiber of $A$ is local with maximal ideal $(0)$, the only prime of
$\widehat A\cong T$ contracting to $(0)\subset A$ is $(0)$ itself; hence
$\mathfrak p\cap A\ne(0)$ for every nonzero prime $\mathfrak p$ of
$\widehat A$. By the criterion that a Noetherian local integral domain
$R$ is weakly quasi-complete iff $\mathfrak p\cap R\ne(0)$ for every
nonzero prime $\mathfrak p$ of $\widehat R$
\cite[Proposition~1]{Farley-QCLocalizations}, $A$ is weakly quasi-complete.

\smallskip
\noindent\emph{Step 4: a homomorphic image of $A$ that fails weak
quasi-completeness.} Consider the height-one prime $Q=(x,y)T$ from
Step~1. Since $Q\ne(0)$, the triviality of the generic formal fiber
gives $\mathfrak q:=Q\cap A\ne(0)$. Faithful flatness of $\widehat A$
over $A$ yields
\[
   1\le\operatorname{ht}(\mathfrak q)\le\operatorname{ht}(Q)=1,
\]
so $\operatorname{ht}(\mathfrak q)=1$. As $A$ is a UFD, $\mathfrak q=aA$ for a
prime element $a\in A$. We claim that $aT$ is not prime in $T$: it is
contained in $Q$ and, if it were prime, the inclusion of height-one primes
would force $aT=Q$, contradicting the non-principality of $Q$.

Therefore $T/aT$ is not a domain. Since completion commutes with
quotient by a single element in a Noetherian local ring,
\[
   \widehat{A/aA}\cong\widehat A/a\widehat A\cong T/aT,
\]
so $\widehat{A/aA}$ is not a domain. As $a$ is prime in the
$2$-dimensional UFD $A$, the quotient $A/aA$ is a $1$-dimensional
Noetherian local domain that is not analytically irreducible. By the
recorded equivalence \cite[Corollary~2.2]{Anderson2014-QCRM} that a $1$-dimensional
Noetherian local domain is (weakly) quasi-complete iff it is
analytically irreducible, $A/aA$ is not weakly quasi-complete.

\smallskip
\noindent\emph{Conclusion.} $A$ is a weakly quasi-complete Noetherian
local ring with a homomorphic image $A/aA$ that is not weakly
quasi-complete. By the equivalence between quasi-completeness and
weak quasi-completeness of all homomorphic images
\cite{Anderson2014-QCRM}, $A$ is not quasi-complete.
\end{proof}

\noindent The above solution has been formalized and machine-verified
in Lean~4; see \cite{JuEtAl2026Rethlas}.


\subsection{Solution to Problem~\ref{prob:21}}\label{subsec:proof-21}

For an integral domain $D$ with fraction field $K=\operatorname{Frac}(D)$,
the ring of \emph{integer-valued polynomials} on $D$ is
\[
   \operatorname{Int}(D):=\{f\in K[X]:f(D)\subseteq D\},
\]
and more generally
$\operatorname{Int}(D^n):=\{f\in K[X_1,\dots,X_n]:f(D\times\cdots\times D)\subseteq D\}$
is the analogous ring of integer-valued polynomials in $n$ variables.
A \emph{$D$-$D$-biring} is a commutative $D$-algebra equipped with a
compatible $D$-coalgebra structure whose comultiplication and counit
are $D$-algebra homomorphisms; the definitions and characterizations
used below are from
\cite{Elliott2010-Plethories}. We use in
particular the notion of a \emph{weakly polynomially complete}
$D$-subalgebra of $\operatorname{Int}(D^n)$ and that of a \emph{weakly
polynomially composite} domain (both in the sense of \emph{loc.\
cit.}). The key fact used at the end of the proof is
\cite[Theorem~12(1)]{Elliott2010-Plethories}: if $\operatorname{Int}(D)$
admits a $D$-$D$-biring structure for which the polynomial inclusion
$D[X]\to\operatorname{Int}(D)$ is a biring homomorphism, then $D$ must
be weakly polynomially composite.

\begin{theorem}\label{thm:21}
There exists an integral domain $D$ such that $\operatorname{Int}(D)$
does not admit any $D$-$D$-biring structure for which the inclusion
$D[X]\to\operatorname{Int}(D)$ is a homomorphism of $D$-$D$-birings.
\end{theorem}

\begin{proof}
We exhibit an explicit domain $D$ over which the conclusion fails. Set
\[
   k:=\mathbf F_2,\quad A:=k[t],\quad S:=A\setminus\bigl((t)\cup(t+1)\bigr),
   \quad T:=S^{-1}A,
\]
and
\[
   N_0:=tT,\quad N_1:=(t+1)T,\quad m:=t(t+1),\quad M:=mT=N_0 N_1,\quad
   D:=k+M\subset T.
\]
Let $v_0$ and $v_1$ be the discrete valuations on $K=\operatorname{Frac}(T)$
extending the valuations of the local rings $T_{N_0}$ and $T_{N_1}$,
normalized by $v_0(t)=1$ and $v_1(t+1)=1$.

\smallskip
\noindent\emph{Step 1: structure of $D$.} The ring $T$ is a semilocal PID
with maximal ideals $N_0$ and $N_1$, and these are comaximal, so
$M=N_0\cap N_1=N_0 N_1=mT$. The composite $D\hookrightarrow T\twoheadrightarrow T/N_i$ has
image $\mathbf F_2$, so $N_i\cap D=M$; in particular $D$ is local with
maximal ideal $M$, residue field $D/M\cong\mathbf F_2$, and $D$ is a
domain. For every $n\ge 1$,
\[
   u_n:=t(t+1)^{n+1}=m(t+1)^n\in M\setminus M^2,
\]
because $v_0(u_n)=1$ while every element of $M^2=m^2 T$ has $v_0$-value
at least~$2$.

\smallskip
\noindent\emph{Step 2: a polynomial $g\in\operatorname{Int}(D)$ with
$g(u_n^2)\notin M^2$.} Define
\[
   q(X):=\frac{X^2+X}{m}\in K[X],\qquad
   g(X):=q(X)^2+q(X)\in K[X],
\]
where $K=\operatorname{Frac}(D)=\operatorname{Frac}(T)$. For $x\in D$ the
residue class of $x$ modulo $M$ lies in $\mathbf F_2$, so $x^2+x\in
M=mT$, whence $q(x)\in T$. Reducing $q(x)\in T$ modulo $N_0$ and
modulo $N_1$ lands in $\mathbf F_2$, so $q(x)^2+q(x)\in N_0\cap N_1=M$;
that is, $g(x)\in M\subset D$. Hence $g\in\operatorname{Int}(D)$ and
$g(D)\subseteq M$. For $u_n^2=m^2(t+1)^{2n}=m^2 w^2$ with $w=(t+1)^n$,
\[
   q(u_n^2)=\frac{u_n^4+u_n^2}{m}=m w^2+m^3 w^4\in M,
\]
and the first summand has $v_0$-value $1$ while the second has
$v_0$-value at least $3$; hence $v_0(q(u_n^2))=1$, i.e.
$q(u_n^2)\in M\setminus M^2$. Since $q(u_n^2)^2\in M^2$, we get
$g(u_n^2)\equiv q(u_n^2)\pmod{M^2}$, so $g(u_n^2)\notin M^2$.

\smallskip
\noindent\emph{Step 3: a difference criterion.} For any finite set
$H\subseteq\operatorname{Int}(D)$ there exists $n\ge 1$ such that
$h(u_n)-h(0)\in M$ for every $h\in H$. Indeed, writing
$h(X)=\sum_{r\ge 0} c_{h,r}X^r$ with $c_{h,r}\in K$, choose $n$ so that
$v_1(c_{h,r})\ge -(n-1)$ for all $h\in H$ and all $r\ge 1$ (possible
because $H$ is finite). Then for $r\ge 1$,
\[
   v_1(c_{h,r}u_n^r)\ge -(n-1)+r(n+1)\ge 2,
\]
so $h(u_n)-h(0)=\sum_{r\ge 1} c_{h,r} u_n^r\in N_1$; as
$h(u_n),h(0)\in D$ this lies in $N_1\cap D=M$.

\smallskip
\noindent\emph{Step 4: $g(XY)\notin\operatorname{im}\theta_2$.} Let
$\theta_2\colon\operatorname{Int}(D)\otimes_D\operatorname{Int}(D)\to
\operatorname{Int}(D^2)$ be the canonical $D$-algebra map, whose image,
under the usual identification of $\operatorname{Int}(D^2)$ with a subring
of $K[X,Y]$, is the set of finite sums $\sum_i f_i(X)h_i(Y)$ with
$f_i,h_i\in\operatorname{Int}(D)$. Set $P(X,Y):=g(XY)$; since $g\in
\operatorname{Int}(D)$ and $XY$ maps $D^2$ into $D$, we have
$P\in\operatorname{Int}(D^2)$.

Suppose, for contradiction, that $P=\sum_{i=1}^N f_i(X)h_i(Y)$ with
$f_i,h_i\in\operatorname{Int}(D)$. Apply Step~3 to
$H=\{f_1,\dots,f_N,h_1,\dots,h_N\}$ to obtain $n\ge 1$ and $u:=u_n$ with
$f_i(u)-f_i(0)\in M$ and $h_i(u)-h_i(0)\in M$ for every $i$. Then
\[
   P(u,u)-P(u,0)-P(0,u)+P(0,0)
   =\sum_{i=1}^N \bigl(f_i(u)-f_i(0)\bigr)\bigl(h_i(u)-h_i(0)\bigr)\in M^2.
\]
On the other hand $g(0)=0$, so $P(u,0)=P(0,u)=P(0,0)=0$ and the left side
equals $g(u^2)$, which is not in $M^2$ by Step~2. This contradiction shows
$P=g(XY)\notin\operatorname{im}\theta_2$.

\smallskip
\noindent\emph{Step 5: conclusion.} Since $XY=\theta_2(X\otimes X)\in
\operatorname{im}\theta_2$ but $g(XY)\notin\operatorname{im}\theta_2$,
the subring $\operatorname{im}\theta_2$ of $\operatorname{Int}(D^2)$ is
not weakly polynomially complete over $D$. By the terminology of
\cite{Elliott2010-Plethories}, $D$ is therefore not
weakly polynomially composite. By \cite[Theorem~12(1)]{Elliott2010-Plethories}, if $\operatorname{Int}(D)$ admitted a
$D$-$D$-biring structure such that $D[X]\to\operatorname{Int}(D)$ is a
biring homomorphism, then $D$ would be weakly polynomially composite.
Hence $\operatorname{Int}(D)$ does not admit such a structure.
\end{proof}

\subsection{Solution to Problem~\ref{prob:35}}\label{subsec:proof-35}

We collect the star-operation notions used below. Let $D$ be an
integral domain with fraction field $K$, and let $J\subseteq K$ be a
fractional ideal of $D$ (i.e.\ a $D$-submodule with
$dJ\subseteq D$ for some $d\in D\setminus\{0\}$). Set
$J^{-1}:=\{x\in K:xJ\subseteq D\}$; the \emph{$t$-closure} of $J$ is
\[
   J_t:=\bigcup_{F\subseteq J\ \text{finitely generated}}(F^{-1})^{-1}.
\]
$J$ is a \emph{$t$-ideal} if $J=J_t$, and a $t$-ideal of \emph{finite
type} if $J=F_t$ for some finitely generated $F$. A \emph{maximal
$t$-ideal} is an integral $t$-ideal maximal among proper integral
$t$-ideals; we write $D_P$ for the localization of $D$ at a (maximal)
$t$-ideal $P$. The fractional ideal $J$ is \emph{$t$-invertible} if
$(JJ^{-1})_t=D$, and is \emph{$t$-locally principal} if $JD_P$ is
principal for every maximal $t$-ideal $P$ of $D$. Two integral
$t$-ideals are \emph{$t$-comaximal} if no maximal $t$-ideal contains
both. The domain $D$ is of \emph{finite $t$-character} if every
nonzero non-unit of $D$ lies in only finitely many maximal $t$-ideals.

$D$ is an \emph{almost GCD} (\emph{AGCD}) \emph{domain} if for every
pair $x,y\in D\setminus\{0\}$ there is an integer $n=n(x,y)\ge 1$ with
$x^n D\cap y^n D$ principal. By
\cite[Lemma~4.3]{Zafrullah-UniversalRestriction}, every AGCD domain is
an \emph{APVMD} (almost Pr\"ufer $v$-multiplication domain --- for every
finite set $x_1,\dots,x_n\in D\setminus\{0\}$ there exists $m\ge 1$
with $(x_1^m,\dots,x_n^m)_t$ a $t$-invertible $t$-ideal) and a
\emph{$t$-SAB domain} (in the sense of \emph{loc.\ cit.}).

\begin{theorem}\label{thm:35}
Let $D$ be an almost GCD domain such that every nonzero $t$-locally
principal ideal of $D$ is $t$-invertible. Then $D$ is of finite
$t$-character.
\end{theorem}

\begin{proof}
The argument combines several results of Zafrullah. We use the
following inputs.

\begin{enumerate}[(i)]
\item Every AGCD domain is an almost Pr\"ufer $v$-multiplication domain
      (APVMD): for every finite set $x_1,\dots,x_n\in D\setminus\{0\}$
      there exists $m\ge 1$ such that $(x_1^m,\dots,x_n^m)_t$ is
      $t$-invertible \cite[Lemma 4.3(6)]{Zafrullah-UniversalRestriction}.
\item Every AGCD domain is a $t$-SAB domain
      \cite[Lemma 4.3(13)]{Zafrullah-UniversalRestriction}.
\item For a $t$-SAB domain $D$ with $\Gamma$ the set of proper nonzero
      principal ideals, $D$ has finite $t$-character if and only if every
      power of every proper $t$-ideal of finite type is contained in at
      most finitely many mutually $t$-comaximal members of $\Gamma$
      \cite[Corollary 6]{Zafrullah-UniversalRestriction}.
\item If $D$ is a domain and there exists an integral $t$-invertible
      $t$-ideal contained in infinitely many mutually $t$-comaximal
      $t$-invertible $t$-ideals of $D$, then $D$ has a nonzero
      $t$-locally principal ideal that is not $t$-invertible
      \cite[Proposition 4]{Zafrullah-BazzoniLike}.
\end{enumerate}

Assume, for contradiction, that $D$ is not of finite $t$-character. By
(ii), $D$ is a $t$-SAB domain, so by (iii) there exist a proper $t$-ideal
$I$ of finite type and a positive integer $m$ such that $I^m$ is
contained in infinitely many mutually $t$-comaximal proper nonzero
principal ideals $\gamma_\lambda$, $\lambda\in\Lambda$ (with $\Lambda$
infinite). Write $I=J_t$ for a finitely generated ideal
$J=(a_1,\dots,a_n)$. Then $a_i^m\in I^m\subseteq\gamma_\lambda$ for every
$i$ and every $\lambda$.

By (i) there exists $r\ge 1$ such that
\[
   C:=(a_1^{mr},\dots,a_n^{mr})_t
\]
is a $t$-invertible $t$-ideal. For each $\lambda\in\Lambda$ the principal
ideal $\gamma_\lambda^r$ is again a $t$-ideal and contains
$a_1^{mr},\dots,a_n^{mr}$, so $C\subseteq\gamma_\lambda^r$. The
$\gamma_\lambda^r$ remain mutually $t$-comaximal: a maximal $t$-ideal
containing both $\gamma_\lambda^r$ and $\gamma_\mu^r$ would contain both
$\gamma_\lambda$ and $\gamma_\mu$, contradicting their mutual
$t$-comaximality.

Thus $C$ is an integral $t$-invertible $t$-ideal contained in infinitely
many mutually $t$-comaximal $t$-invertible $t$-ideals $\gamma_\lambda^r$
of $D$. By (iv), $D$ has a nonzero $t$-locally principal ideal that is
not $t$-invertible, contradicting the hypothesis on $D$.

Hence $D$ is of finite $t$-character.
\end{proof}

\subsection{Solution to Problem~\ref{prob:37b}}\label{subsec:proof-37b}

A finite-dimensional integral domain $D$ is a \emph{Jaffard domain} if
$\dim D[X_1,\dots,X_n]=n+\dim D$ for every $n\ge 1$; equivalently, if
$\dim D=\dim_v D$, where the \emph{valuative dimension} $\dim_v D$ is
the supremum of $\dim V$ over valuation overrings $V$ of $D$ in
$\operatorname{Frac}(D)$. $D$ is \emph{locally Jaffard} if $D_P$ is
Jaffard for every prime ideal $P$ of $D$. The domain under
consideration is David's $3$-dimensional non-Noetherian factorial
domain
\[
   J=\bigcup_{n\ge 1} J_n,\qquad J_n=k[X,\zeta_{n-1},\zeta_n],
\]
where $k$ is a characteristic-zero field, $\{s(n)\}_{n\ge 2}$ is a
sequence of positive integers, and the elements $\zeta_n$ satisfy the
recurrence $X\zeta_n=\zeta_{n-1}^{s(n)}+\zeta_{n-2}$ for $n\ge 2$. We
use freely the following properties of $J$ established in
\cite{David1973}: $J$ is a UFD of Krull dimension $3$; $XJ$ is prime;
$J_{(X)}$ is a rank-one discrete valuation ring (DVR); and the
``origin'' ideal $(X,\zeta_0,\zeta_1,\dots)J$ is a maximal ideal of
height three.

\begin{theorem}\label{thm:37b}
Let $J=\bigcup_{n\ge 1}J_n$ with $J_n=k[X,\zeta_{n-1},\zeta_n]$, over a
characteristic-zero field $k$, with the recurrence
$X\zeta_n=\zeta_{n-1}^{s(n)}+\zeta_{n-2}$ for $n\ge 2$, be David's
$3$-dimensional non-Noetherian factorial domain \cite{David1973}. Then
$J$ is locally Jaffard.
\end{theorem}

\begin{proof}
We show that $J_P$ is Jaffard for every $P\in\operatorname{Spec}(J)$.
By \cite[Theorem~2.3]{DobbsFontanaKabbaj1990} applied to the directed
union $J_n\subseteq J\subseteq J_n[X^{-1}]$ of affine Jaffard domains,
$J$ itself is Jaffard, hence
\[
   \dim_v(J)=\dim(J)=3.
\]

\smallskip
\noindent\emph{Case 1: $X\notin P$.} Then $J_P=(J[X^{-1}])_{PJ[X^{-1}]}
=(J_n[X^{-1}])_{PJ_n[X^{-1}]}$ for any $n$, which is a localization of a
three-variable polynomial ring over $k$, hence Noetherian and thus
Jaffard.

\smallskip
\noindent\emph{Case 2: $X\in P$, $\operatorname{ht}(P)=1$.} Since $XJ$ is
prime by \cite{David1973}, this forces $P=XJ$. By \cite[Lemma~2.3]{David1973}, $J_P=J_{(X)}$ is a rank-one discrete valuation ring,
in particular Jaffard.

\smallskip
\noindent\emph{Case 3: $X\in P$, $\operatorname{ht}(P)=3$.} Then
$\dim(J_P)=3$ and, since localization does not increase valuative
dimension,
\[
   \dim_v(J_P)\le\dim_v(J)=3.
\]
The reverse inequality $\dim_v(J_P)\ge\dim(J_P)=3$ always holds, so
$J_P$ is Jaffard.

\smallskip
\noindent\emph{Case 4: $X\in P$, $\operatorname{ht}(P)=2$.} This is the
substantive case. Let $A:=J/XJ$ and write $b_n$ for the image of
$\zeta_n$ in $A$. Modulo $X$, the recurrence becomes
$b_{n-2}=-b_{n-1}^{s(n)}$, so the image of $J_n=k[X,\zeta_{n-1},\zeta_n]$
in $A$ is $k[b_n]$, the transition $k[b_n]\hookrightarrow k[b_{n+1}]$
sends $b_n$ to $-b_{n+1}^{s(n+2)}$, and $A=\bigcup_n k[b_n]$.

Set $\mathfrak q:=P/XJ\subset A$. If $\mathfrak q=0$ then
$P=XJ$, contradicting $\operatorname{ht}(P)=2$. If $\mathfrak q$
contained some $b_n$, then the recurrence would force it to contain
every $b_m$, so $\mathfrak q=(b_0,b_1,\dots)$ and its inverse image in
$J$ would be David's origin maximal ideal
$(X,\zeta_0,\zeta_1,\dots)J$, which has height three by
\cite{David1973}, contradicting $\operatorname{ht}(P)=2$. Hence
$\mathfrak q$ is nonzero and contains no $b_n$.

For each $n$ the contraction $\mathfrak q_n:=\mathfrak q\cap k[b_n]$ is
a nonzero prime of $k[b_n]$: integrality of all transition maps makes
later contractions nonzero once one is, and earlier contractions cannot
be zero because otherwise the field $k[b_m]/\mathfrak q_m$ would be
integral over the non-field $k[b_n]$. Therefore $\mathfrak q_n$ is
maximal in $k[b_n]$; write $\mathfrak q_n=(f_n(b_n))$ with
$f_n(0)\ne 0$. Since $A/\mathfrak q=\bigcup_n k[b_n]/\mathfrak q_n$ is a
directed union of fields along field embeddings, $A/\mathfrak q$ is a
field, so $\mathfrak q$ is maximal. Setting $R_n:=k[b_n]_{\mathfrak q_n}$,
the local map $R_n\to R_{n+1}$ is finite and unramified at
$\mathfrak q_{n+1}$ because the defining equation is
$b_n+b_{n+1}^{s(n+2)}=0$ and its derivative
$s(n+2)\,b_{n+1}^{s(n+2)-1}$ is a unit (since the characteristic is
zero and $b_{n+1}\notin\mathfrak q_{n+1}$). Therefore a uniformizer of
some $R_n$ remains a uniformizer in every later $R_m$, and
$A_{\mathfrak q}=\bigcup_n R_n$ is a DVR.

Now set $S:=J_P$ and $\mathfrak m:=PS$. The contraction of $P$ to $J_n$
is $P_n=(X,\zeta_{n-1}+\zeta_n^{s(n+1)},f_n(\zeta_n))J_n$, the inverse
image of $\mathfrak q_n$ under $J_n\to A$. Choose an index $N$ and put
$t:=f_N(\zeta_N)\in S$. Every element of $P$ lies in some $P_n$. The
generator $\zeta_{n-1}+\zeta_n^{s(n+1)}=X\zeta_{n+1}\in XS$. The image
of $f_n(\zeta_n)$ in $S/XS\cong A_{\mathfrak q}$ is a uniformizer of the
DVR $A_{\mathfrak q}$, as is the image of $t$, and they therefore differ
by a unit of $A_{\mathfrak q}$; lifting that unit to a unit of $S$ gives
$f_n(\zeta_n)\in(X,t)S$. Therefore $\mathfrak m=(X,t)S$.

We claim that every prime of $S$ is finitely generated, which by
Cohen's theorem implies $S$ is Noetherian. The zero ideal and the
maximal ideal $\mathfrak m=(X,t)S$ are visibly finitely generated. If
$H$ is a nonzero non-maximal prime of $S$, set $H_0:=H\cap J$. Then
$H=H_0 S$ and $H_0\subsetneq P$. Since $\operatorname{ht}_J(P)=2$, the
strict inclusion $H_0\subsetneq P$ forces $\operatorname{ht}_J(H_0)=1$
(else a length-two chain below $H_0$ followed by $H_0\subsetneq P$ would
push $\operatorname{ht}_J(P)\ge 3$). Hence $H_0$ is a height-one prime of
the factorial domain $J$ \cite{David1973}, so generated by a prime
element; localizing at $P$ shows $H$ is principal in $S$. Thus all
primes of $S$ are finitely generated and $S$ is Noetherian. A
finite-dimensional Noetherian domain is Jaffard, so $J_P$ is Jaffard.

This completes Case~4 and the proof.
\end{proof}

\subsection{Solution to Question~\ref{q:bs61}}\label{subsec:proof-bs61}

Let $k$ be a field. For a finitely generated graded module $M$ over a
graded Noetherian $k$-algebra $R$ with $R_0=k$, the \emph{Betti table}
of $M$ has entries
$\beta_{i,j}(M):=\dim_k\operatorname{Tor}_i^R(M,k)_j$. A strictly
increasing tuple $(d_0,d_1,\dots,d_n)$ of integers is a \emph{degree
sequence}; the corresponding \emph{pure ray} in the Boij--S\"oderberg
cone is the half-line of Betti tables supported in degrees $(i,d_i)$
--- i.e.\ with $\beta_{i,j}=0$ for $j\ne d_i$ --- whose nonzero entries
$\beta_i:=\beta_{i,d_i}$ satisfy the \emph{Herzog--K\"uhl equations}
\[
   \sum_{i=0}^n (-1)^i \beta_i d_i^m=0,\qquad m=0,1,\dots,n-1.
\]
An \emph{integral point} on the ray is a tuple
$(\beta_0,\dots,\beta_n)$ of nonnegative integers satisfying these
equations. The equations have a one-dimensional solution space, with
primitive integral point proportional to
$\bigl(\prod_{j\ne i}|d_i-d_j|^{-1}\bigr)_{i=0,\dots,n}$.

Throughout this subsection, $S=k[x,y,z]$ carries the standard grading,
and $\mathfrak H$ is the Heisenberg Lie algebra on basis $\{x,y,z\}$
with $[x,y]=z$, $[x,z]=[y,z]=0$, graded by $\deg(x)=\deg(y)=1$,
$\deg(z)=2$; its universal enveloping algebra $U(\mathfrak H)$
inherits this $\mathbf Z_{\ge 0}$-grading.

\begin{theorem}\label{thm:bs61}
Let $k$ be a field, let $S=k[x,y,z]$ with $\deg(x)=\deg(y)=\deg(z)=1$,
and let $\mathfrak H$ be the Heisenberg Lie algebra on basis $\{x,y,z\}$
with $[x,y]=z$, $[x,z]=[y,z]=0$, graded by $\deg(x)=\deg(y)=1$,
$\deg(z)=2$. There exist a degree sequence $(d_0,d_1,d_2,d_3)$ and an
integral point on the corresponding pure ray in the Boij--S\"oderberg
cone that arises as the Betti table of no finite length graded module
over $S$ and of no finite length graded module over $U(\mathfrak H)$.
\end{theorem}

\begin{proof}
Take the degree sequence
\[
   (d_0,d_1,d_2,d_3)=(0,6,20,21).
\]
The codimension-$3$ Herzog--K\"uhl ratios for this pure ray are
\[
   \frac{1}{6\cdot 20\cdot 21}:\frac{1}{6\cdot 14\cdot 15}
   :\frac{1}{20\cdot 14\cdot 1}:\frac{1}{21\cdot 15\cdot 1}
   =1:2:9:8,
\]
so $\beta=(1,2,9,8)$ is an integral point on that ray. We show $\beta$ is
realizable over neither $S$ nor $U(\mathfrak H)$.

\smallskip
\noindent\emph{Not realizable over $S$.} If a finite length graded
$S$-module $M$ had Betti table $\beta$, then $\beta_0(M)=1$, so
$M\cong S/I$ for some homogeneous ideal $I$, and $\beta_1(M)=2$ means
$I$ is generated by two elements. By Krull's height theorem
$\operatorname{ht}(I)\le 2$, so $\dim(S/I)\ge 1$, contradicting finite
length.

\smallskip
\noindent\emph{Not realizable over $A:=U(\mathfrak H)$.} By the
Poincar\'e--Birkhoff--Witt theorem, the monomials $x^a y^b z^c$ form a
graded $k$-basis of $A$ with $\deg(x)=\deg(y)=1$, $\deg(z)=2$, so
\[
   H_A(t)=\sum_{a,b,c\ge 0}t^{a+b+2c}=\frac{1}{(1-t)^2(1-t^2)}.
\]
If a finite length graded $A$-module $N$ had Betti table $\beta$, then
\[
   H_N(t)=\frac{1-2t^6+9t^{20}-8t^{21}}{(1-t)^2(1-t^2)},
\]
and $H_N(t)$ must be a polynomial. Hence
$P(t):=1-2t^6+9t^{20}-8t^{21}$ must be divisible by $(1-t)^2(1-t^2)$,
and in particular by $1+t$. However,
\[
   P(-1)=1-2+9+8=16\ne 0,
\]
so $1+t$ does not divide $P(t)$, contradiction.

Therefore the integral point $(1,2,9,8)$ on the pure ray for
$(0,6,20,21)$ is realizable over neither $S$ nor $U(\mathfrak H)$.
\end{proof}

\subsection{Solution to Question~\ref{q:bs62}}\label{subsec:proof-bs62}

Let $k$ be a field. A \emph{positively graded Lie algebra} over $k$ is
$\mathfrak g=\bigoplus_{i\ge 1}\mathfrak g_i$ with bracket of degree
zero; it is \emph{generated in degree~$1$} if it is generated as a Lie
algebra by $\mathfrak g_1$. The universal enveloping algebra
$U(\mathfrak g)$ inherits a $\mathbf Z_{\ge 0}$-grading, and by the
Poincar\'e--Birkhoff--Witt theorem the ordered monomials in a
homogeneous basis of $\mathfrak g$ form a graded $k$-basis of
$U(\mathfrak g)$. For a finitely generated graded $U(\mathfrak g)$-module
$M$, the \emph{Betti table} of $M$ has entries
$\beta_{i,j}(M):=\dim_k\operatorname{Tor}_i^{U(\mathfrak g)}(M,k)_j$.
A \emph{degree sequence} is a strictly increasing tuple
$(d_0,\dots,d_n)$ of integers; the corresponding \emph{pure ray} in
the Boij--S\"oderberg cone is the half-line of Betti tables supported
in degrees $(i,d_i)$, with nonzero entries
$\beta_i:=\beta_{i,d_i}$ satisfying the \emph{Herzog--K\"uhl equations}
$\sum_{i=0}^n(-1)^i\beta_i d_i^m=0$ for $m=0,1,\dots,n-1$. An
\emph{integral point} on the ray is a tuple of nonnegative integers
satisfying these equations.

\begin{theorem}\label{thm:bs62}
Let $k$ be a field. There exist an integer $n\ge 1$, a degree sequence
$(d_0,\dots,d_n)$, and an integral point on the corresponding pure ray
in the Boij--S\"oderberg cone for which no $n$-dimensional
$\mathbb Z_{>0}$-graded Lie algebra $\mathfrak g$ generated in degree~$1$
over $k$ admits a finite length graded module $M$ over $U(\mathfrak g)$
whose Betti table is that integral point.
\end{theorem}

\begin{proof}
We give a counterexample with $n=4$ and
\[
   d=(0,1,4,5,6).
\]

\smallskip
\noindent\emph{Step 1: the primitive integral point.} Let
$\beta=(\beta_0,\beta_1,\beta_2,\beta_3,\beta_4)$ be the coefficients of
a pure table on this ray, so the only nonzero Betti entries are
$\beta_{i,d_i}=\beta_i$. The Herzog--K\"uhl equations for a codimension-$4$
pure resolution read
\[
   \sum_{i=0}^4 (-1)^i \beta_i d_i^m=0,\qquad m=0,1,2,3.
\]
Normalizing $\beta_0=1$, the four equations become
\[
   1-\beta_1+\beta_2-\beta_3+\beta_4=0,\qquad
   -\beta_1+4\beta_2-5\beta_3+6\beta_4=0,
\]
\[
   -\beta_1+16\beta_2-25\beta_3+36\beta_4=0,\qquad
   -\beta_1+64\beta_2-125\beta_3+216\beta_4=0.
\]
Subtracting consecutive equations,
\[
   6\beta_2-10\beta_3+15\beta_4=0,\qquad
   12\beta_2-25\beta_3+45\beta_4=0;
\]
subtracting twice the first from the second gives $\beta_3=3\beta_4$, and
back-substitution yields $2\beta_2=5\beta_4$ and $\beta_1=\beta_4$.
Plugging into the first equation gives $\beta_4=2$, so
\[
   (\beta_0,\beta_1,\beta_2,\beta_3,\beta_4)=(1,2,5,6,2),
\]
which is already integral, hence the primitive integral point on the
ray.

\smallskip
\noindent\emph{Step 2: PBW Hilbert series and divisibility.} Let
$\mathfrak g=\bigoplus_{i\ge 1}\mathfrak g_i$ be a finite-dimensional
positively graded Lie algebra over $k$ with $h_i:=\dim_k\mathfrak g_i$.
Choosing a homogeneous basis and applying PBW, the ordered monomials in
basis vectors form a graded $k$-basis of $U(\mathfrak g)$, so
\[
   H_{U(\mathfrak g)}(t)=\prod_{i\ge 1}(1-t^i)^{-h_i}.
\]
If a finite length graded $U(\mathfrak g)$-module $M$ admits a pure
resolution with Betti numbers $\beta_i$ in degrees $d_i$, then taking
alternating sums of Hilbert series in a graded free resolution gives
$H_M(t)=p_M(t)\,H_{U(\mathfrak g)}(t)$ with $p_M(t)=\sum_i(-1)^i\beta_i
t^{d_i}$. Since $H_M(t)$ is a polynomial,
\[
   p_M(t) \text{ is divisible by } \prod_{i\ge 1}(1-t^i)^{h_i}.
\]

\smallskip
\noindent\emph{Step 3: contradiction.} Suppose $\mathfrak g$ is a
$4$-dimensional positively graded Lie algebra over $k$ generated in
degree~$1$ and $M$ is a finite length graded $U(\mathfrak g)$-module with
the Betti table from Step~1. The alternating Betti polynomial is
\[
   p(t)=1-2t+5t^4-6t^5+2t^6=(1-t)^4\bigl(2t^2+2t+1\bigr).
\]
By Step~2 it is divisible by $\prod_{i\ge 1}(1-t^i)^{h_i}$. Since
$\dim\mathfrak g=\sum h_i=4$, dividing by $(1-t)^4$ leaves
\[
   q(t):=2t^2+2t+1
\]
divisible by $\prod_{i\ge 2}(1+t+\cdots+t^{i-1})^{h_i}$. But:
$q(-1)=1\ne 0$, so $1+t$ does not divide $q$; modulo $1+t+t^2$ one has
$t^2+t\equiv -1$, hence $q(t)\equiv -1$, so $1+t+t^2$ does not divide
$q$ over any field; and $\deg q=2$, so no $1+t+\cdots+t^{i-1}$ with
$i\ge 4$ can divide $q$. Therefore $h_i=0$ for all $i\ge 2$, i.e.
$\mathfrak g$ is concentrated in degree~$1$. Then
$[\mathfrak g,\mathfrak g]\subseteq\mathfrak g_2=0$, so $\mathfrak g$ is
abelian and
\[
   U(\mathfrak g)\cong k[x_1,x_2,x_3,x_4]=:R.
\]
So $M$ would be a finite length graded $R$-module with Betti table
$(1,2,5,6,2)$. But $\beta_0=1$ means $M\cong R/I$ for some homogeneous
ideal $I$, and $\beta_1=2$ means $I$ is minimally generated by two
homogeneous elements; by Krull's height theorem $\operatorname{ht}(I)\le 2$,
so $\dim(R/I)\ge 4-2=2$, contradicting finite length.

Thus no such pair $(\mathfrak g,M)$ exists for the degree sequence
$(0,1,4,5,6)$ and the integral point $(1,2,5,6,2)$.
\end{proof}


\end{document}